\documentclass{article}
\usepackage {amsmath, amsfonts, amssymb, mathrsfs, amsthm,stmaryrd,sectsty,hyphenat,enumitem,calc,url,authblk}
\usepackage{oscar}
\usepackage{setspace}
\newcommand{\cOl}{{\mathcal O}_{\ell}}

\usepackage[normalem]{ulem}
\usepackage{rank-2-roots}
\usepackage{chngcntr}
\counterwithin{figure}{section}
\begin{document}
\title{On the affine Springer fibers inside the invariant center of the small quantum group\vspace{-2ex}}
\author[1]{Nicolas Hemelsoet}
\author[2]{Oscar Kivinen}
\author[1]{Anna Lachowska}
\affil[1]{Department of Mathematics, EPFL  \thanks{nicolas.hemelsoet@epfl.ch, anna.lachowska@epfl.ch}}
\affil[2]{Department of Mathematics and Systems Analysis, Aalto University\thanks{oscar.kivinen@aalto.fi}\vspace{-2ex}}
\date{\today\vspace{-3ex}}
\maketitle
\begin{abstract} 
Let $\fu_\zeta^\vee$ denote the small quantum group associated with a simple Lie algebra $\fg^\vee$ and a root of unity $\zeta$. In \cite{BBASV} a geometric realization of $Z(\fu_\zeta^\vee)^{G^\vee}$, the $G^\vee$-invariant part of the center of $\fu_\zeta^\vee$, was proposed. We compute the dimension of the geometric subalgebra of the center and in the case where $G=SL_n$, we study a bigraded refinement of the result.
\end{abstract} 
\setcounter{tocdepth}{2}
\begin{spacing}{0.3}
\tableofcontents 
\end{spacing}
\section{Introduction} 

Let $G$ be a complex simple simply connected algebraic group, and $\fg$ its Lie algebra. We will fix Cartan and Borel subalgebras $\ft \subset \fb \subset \fg$. Denote also by $\ft^\vee\subset\fb^\vee\subset\fg^\vee$ the same data for the Langlands dual Lie algebra. Denote $T \subset B$ (respectively $G^{\vee}$) the connected Lie groups corresponding to $\ft \subset \fb$ (respectively $\fg^{\vee}$). Let $\ell$ be an odd number that is greater than the Coxeter number $h$ of $\fg$ and coprime to the determinant of the Cartan matrix and to $h, h+1$. 

We denote by $\fu_\zeta^\vee=\fu_\zeta(\fg^\vee)$ the small quantum group associated to the Lie algebra \cite{Lus}.
Let $\Lambda$ denote the cocharacter lattice of $G$, 
and $W$ the Weyl group of $\fg$ (and $\fg^\vee$). Then $\fu^\vee_\zeta$ decomposes into a direct sum of blocks enumerated by the orbits of the extended affine Weyl group $\tW = W \ltimes \Lambda$ acting via the $\ell$-dilated dot action on $\Lambda$. See for example the introduction to \cite{LQ2}. We denote by $\fu_\zeta^{\vee,\lambda}$ the block corresponding to the $\tW$-orbit of $\lambda \in \Lambda$. In particular, $\fu_\zeta^{\vee,0}$ denotes the principal (regular) block of the small quantum group. 

In this paper we derive a dimension formula for a geometrically defined subalgebra $Z(u^\vee_\zeta)^{G^\vee}_{geom}\subseteq Z(u^\vee_\zeta)^{G^\vee}$, the $G^\vee$-invariant part of the center of the small quantum group. The precise definition is in terms of affine Springer theory and will be given in due course. One of our main results is a dimension formula for this subalgebra:
\begin{theorem} 
\label{thm:mainthm}
Suppose that $\ell$ is as above. Then 
$${\rm dim}\, Z(u^\vee_\zeta)^{G^\vee}_{geom} ={\rm Cat}_W((h+1)\ell-h,h) , $$
where ${\rm Cat}_W$ is the generalized rational Coxeter--Catalan number of $W$, and $h$ the Coxeter number associated with the root system of $\fg$. 
\end{theorem} 
After writing the present paper, we received the complete text of \cite{BBASV}, where the same dimension formula is obtained in \cite[Theorem C]{BBASV}. Whereas the proof in \cite{BBASV} uses a comparison to certain elliptic affine Springer fibers and the results of Sommers, our argument is based on the block decomposition and Coxeter--Catalan combinatorics. However, it will be clear to the reader that our results rely heavily on the definitions and results in \cite{BBASV}.
A major point of divergence from that paper is the case $G=SL_n$, for which we study a refinement of Theorem \ref{thm:mainthm} involving a bigrading coming from the geometry of affine Springer fibers. This is done in Sections \ref{sec:dr} and \ref{sec:hitchin}.

We now recall the definition of $Z(u^\vee_\zeta)^{G^\vee}_{geom}$, following \cite{BBASV}.
Let $\Gr^{\gamma,\zeta} = \Gr^{\gamma} \cap \Gr^{\zeta}$, where $\Gr^{\gamma}$ is the affine Springer fiber with $\gamma = t^{\ell-1}s$ for a regular element $s \in {\ft}^{reg}$, and $\Gr^{\zeta}$ is the set of fixed points for the cyclic group action generated by $\zeta$. 
There is a left $\widetilde{W}$-action on the singular cohomology groups $H^*(\Gr^{\gamma,\zeta})$ induced by the lattice action on the affine Springer fiber, as well as a monodromy action coming from variation of $s$ in a family. This action is called the equivariant centralizer-monodromy action in \cite{BAL} and the dot action in \cite{CM}. More precise definitions will be given in the next section.
\begin{theorem}[\cite{BBASV}] \label{BBASV-thm} 
There is an algebra embedding 
$$H^*(\Gr^{\gamma,\zeta})^{\widetilde{W}\cdot} \subseteq  Z(\fu^\vee_\zeta)^{G^\vee} , $$
where the product on the left is the cup product and we take invariants for the dot action.
\end{theorem}
Motivated by this, we let $Z(u^\vee_\zeta)^{G^\vee}_{geom}:=H^*(\Gr^{\gamma,\zeta})^{\widetilde{W}\cdot}$. Conjecturally, this is all of the $G^\vee$-invariant center:
\begin{conjecture}[\cite{BBASV}] \label{BBASV-conj} 
The embedding above is an isomorphism. 
\end{conjecture} 

The logical structure of the proof of the dimension formula in Theorem \ref{thm:mainthm} is as follows. Both sides of Theorem \ref{BBASV-thm} naturally decompose into blocks. Geometrically, this follows from the fact that $\Gr^{\gamma,\zeta}$ can be written as a finite disjoint union of generalizations of affine Springer fibers, as explained in the next section. This block decomposition respects the embedding into $Z(\fu^\vee_\zeta)^{G^\vee}$. Crucially, reinterpreting $\Gr^{\gamma,\zeta}$ using varieties which do not depend on $\ell$ lets us describe each block without using the variable $\ell$.

First restricting to the regular block, we identify $Z(u^\vee_\zeta)^{G^\vee}_{geom}$ as a $W$-representation with $\C[\Lambda/(h+1)\Lambda]$. We then rewrite all of $H^*(\Gr^{\gamma,\zeta})^{\widetilde{W}\cdot}$ by relating the singular blocks with the regular block. Both of these constructions use a second $W$-action on the cohomology, given by the affine Springer action, which is not obvious from the point of view of $Z(u_\zeta^\vee)^{G^\vee}$. 

Assuming Conjecture \ref{BBASV-conj}, the regular block part of Theorem \ref{thm:mainthm} agrees with an ungraded form of the conjecture formulated in \cite{LQ1}. 
In the case where $\fg = {\mathfrak{sl}}_n$, the formula of Theorem \ref{thm:mainthm} coincides with the formula conjectured by Igor Frenkel for the whole $G^\vee$-invariant part (computed for $n\leq 4$ in \cite{LQ2}):  
\begin{corollary}
\label{coro:catalan}
Let $G = SL_n$, and suppose that $n\not\equiv 0, -1$ mod $\ell$. Then  $${\rm dim}\, H^*(\Gr^{\gamma,\zeta})^{\widetilde{W}\cdot}={\rm Cat}_n((n+1)\ell-n,n)=\frac{1}{(n+1)\ell}\binom{(n+1)\ell}{n},$$ the rational $((n+1)\ell-n,n)$-Catalan number.
\end{corollary}
The $G^\vee$-invariant parts of the blocks of
the center, indeed the entire blocks, can be equipped with two gradings that arise from
isomorphisms with the equivariant cohomologies of certain coherent sheaves on the
Springer resolution\cite{BL} (see Theorem \ref{block:coherent}). 

On the geometric side, the bigrading of the principal block is closely related to the isomorphism of $W-$representations $\C[\Lambda/(h+1)\Lambda]\otimes \sgn \cong \overline{DR}_W$, where the latter is Gordon's canonical quotient of the space of the diagonal coinvariants \cite{Gordon}. The latter space is algebraically defined and has an obvious bigrading.

The bigrading in the type A case can also be realized on the side of the affine Springer fibers by passing from $G=SL_n$ to the (non-simply connected) group $G=GL_n$. As we show in Theorem \ref{thm:bgrd}, the resulting bigrading coincides with that of the sign-twisted diagonal coinvariants. This geometric bigrading comes from the realization of the invariant piece of the cohomology as a quotient of the BM homology of the ''positive part" of the affine Springer fiber for $GL_n$, up to a linear dual.

In Section \ref{sec:hitchin} we also study another model for the bigrading, coming from the perverse filtration on certain parabolic Hitchin fibers of $G=SL_n$. This formulation is more geometric and allows us, for example,  to define an ${\mathfrak{sl}}_2$-action  on the blocks of $H^*(\Gr^{\gamma,\zeta})^{\tW\cdot}$. Based on the Hilb-vs-Quot conjecture of \cite{KiTr}, the two bigradings are the same up to a simple change of variables.

To complete the triangle between affine Springer fibers, coherent sheaves on the Springer resolution, and the diagonal coinvariants, we conjecture that the transport of either of our two bigradings to the blocks of the center coincides with the one coming from the equivariant cohomologies of the coherent sheaves on the Springer resolution. There is a natural  $\fs\fl_2$-action on the cohomologies on the affine Springer fiber side, which we conjecture to coincide with the ${\mathfrak{sl}}_2$ action ''along the diagonals" as in \cite[Section 4]{LQ1}. 

The outline of the paper is as follows. 
In Section \ref{sec:asf}, we define the geometric objects and prove Theorem \ref{thm:mainthm}. In Section \ref{sec:dr}, we study the diagonal coinvariant ring with a focus on type A, providing $H^*(\Gr^{\gamma,\zeta})$ with a combinatorially defined bigrading. In Section \ref{sec:smallquantum}, we recall some earlier results on the structure of the center of the small quantum group and apply results from the previous two sections to the structure of the Harish-Chandra center, the Higman ideal and the Verlinde quotient of $\fu_\zeta^\vee$. In Section \ref{sec:hitchin}, we study a geometrically defined bigrading on $H^*(\Gr^{\gamma,\zeta})$ and its relation to the center together with the bigrading defined in Section \ref{sec:dr}.

\subsection*{Acknowledgements}
We would like to thank Erik Carlsson and Peng Shan for helpful explanations of their work, and Roman Bezrukavnikov for sharing the paper \cite{BBASV} with us. Additionally, we would like to thank the referee for many helpful suggestions that significatly improved the text. 

N.H. appreciates the support of the grant number 182767 of the Swiss National Science Foundation. O.K. was supported by a project grant from the Väisälä Foundation.
A.L. is grateful for the support from Facult\'e des Sciences de Base at \'Ecole Polytechnique F\'ed\'erale de Lausanne.

\section{Affine Springer fibers}
\label{sec:asf}
\subsection{Definitions}
In this section, we define the geometric objects under study. For more details, the reader is referred to \cite[Section 4]{RW}.
Let $\zeta$ be a primitive $\ell$-th root of unity and $s\in \ft^{reg}$. Let $\cO:=\C[[t]], \cK:=\C((t))$, $\cOl := \C[[t^\ell]]$ and $\cK_\ell = \C((t^{\ell}))$. Let $\G_m^{rot}$ be the loop rotation torus. It acts on $\cO, \cK$ via $t\mapsto z  t$, as well as on $G(\cO)$ and $G(\cK)$. Let $\mu_\ell=\ker(\G_m^{rot}\xrightarrow{z \mapsto z^\ell} \G_m^{rot})$ be the subgroup of $\ell$th roots of unity.

Let $R$ be a $\C$-algebra. The fppf sheafification of the functor $$R\mapsto G(R((t)))/G(R[[t]])$$ is an ind-scheme called the affine Grassmannian which we will denote by $\Gr$.  More generally, we can define partial affine flag varieties for $G(\cK)$ as follows. The choice of $T$ fixes a standard apartment $A$ in the Bruhat--Tits building of $G(\cK)$. To any facet $\mathbf{f}\subset A$ there is an associated parahoric subgroup $\bP_\mathbf{f}$. 

For a given parahoric subgroup $\bP$, we write $\Fl_{\bP}$ for the corresponding partial affine flag variety. This is also an ind-scheme which coincides with $\Gr$ when $\bP=G(\cO)$; when $\bP=\bI$ is the standard Iwahori defined below, we simply write $\Fl$. We will also consider partial affine flag varieties and affine Springer fibers with the loop variable $t^\ell$ in place of $t$. Given a parahoric subgroup $\bP\subset G(\cK_\ell)$ we write $\Fl_{\bP}^\ell$  for the corresponding partial affine flag variety and so on; for example $\Gr^\ell$ is the fppf sheafification of the functor $$R\mapsto G(R((t^\ell)))/G(R[[t^\ell]]) .$$

The $\mu_\ell$-action on $G(\cK)$ descends to $\Gr$ and the fixed point ind-scheme can be described as follows. Fix a fundamental alcove $\mathbf{a} \subset A$ for the $\ell$-dilated dot action of the (extended) affine Weyl group. If $\lambda \in \Lambda\cap\overline{\mathbf{a}}$ is a cocharacter, we let $\bP_\lambda=\bP_{\mathbf{f}}$ for the unique facet $\mathbf{f}$ containing $\lambda$. For example when $\lambda$ is a regular weight, this is the standard Iwahori subgroup $\bI$, which can also be described as $\bI = ev_0^{-1}(B)$, where $ev_0 : G(\cO) \to G$ is the evaluation map induced by $t \mapsto 0$. As in \cite[§1.9]{RW}, a set of representatives for $\Lambda/\tW$  can be indexed by cocharacters in the closure of the $\ell$-dilated fundamental alcove. We will use this parametrization throughout. From \cite[Proposition 4.7]{RW} we have
\begin{proposition}
There is an isomorphism of ind-schemes
$$\Gr^{\mu_\ell}=\bigsqcup_{\lambda \in \Lambda/\widetilde{W}} \Fl^\ell_{\bP_\lambda} \; ,$$ where $\bP_\lambda$ is the parahoric group scheme in $G(\cK_\ell)$  associated to $\lambda$.  Here the extended affine Weyl group $\widetilde{W}$ is acting by the $\ell$-dilated action on $\Lambda$.
\end{proposition}
For $\zeta$ a generator of $\mu_\ell$, we will write $\Gr^\zeta$ in place of   $\Gr^{\mu_\ell}$. 

Given $\gamma\in\fg(\cK)$, we denote by $$\Fl_{\bP}^\gamma=\{g\bP|\Ad_{g^{-1}}\gamma\in\Lie(\bP)\}$$ the affine Springer fiber of $\gamma$ inside $\Fl_\bP=G(\cK)/\bP$. Similarly, we denote 
$$\ffl_{\bP}^\gamma=\{[g]\in \Fl_{\bP_{\lambda}}|\Ad_{g^{-1}} \gamma \in \Lie(\text{Rad}(\bP))\},$$ where $\text{Rad}(\bP)$ denotes the pro-unipotent radical. The latter (ind-)varieties are generalizations of affine Springer fibers, termed ''affine Spaltenstein fibers" in \cite{BBASV}. Throughout this paper, we only use the underlying reduced ind-schemes of these spaces.

Let now $\gamma=t^{\ell-1}s$ and consider 
$\Gr^{\gamma,\zeta}=\Gr^\gamma\cap\Gr^{\zeta}$.
\begin{proposition}
\label{prop:affspalt}
We have an isomorphism of (ind-)varieties \begin{equation}
\label{eq:spfixedpoints}
\Gr^{\gamma,\zeta}=\bigsqcup_{\lambda \in \Lambda/\widetilde{W}} \ffl_{\bP_\lambda}^{t\gamma,\ell} \; .\end{equation} 
\end{proposition}
\begin{proof}
We work on closed points. Note that $$\Gr^\gamma\cong \{g\in \Gr|\Ad_{g^{-1}} s t^{\ell} \in t\fg(\cO)\}.$$ 
Moreover, we have that $G(\cK)^{\zeta}=G(\cK_\ell)$. 
Since the actions of $\zeta$ and $\gamma$ commute, 
\begin{align*} \Gr^{\gamma,\zeta} &=\bigsqcup_{\lambda \in \Lambda/\tW} \; \{g\in \Fl_{\bP_\lambda}^\ell|\Ad_{g^{-1}}t\gamma\in \fg(\cK_\ell)\cap \Ad_{t^{-\lambda}}(t\fg(\cO))\} \\ 
&= \; \bigsqcup_\lambda \; \{g\in \Fl^\ell_{\bP_{\lambda}}|\Ad_{g^{-1}} t\gamma \in \Lie(\text{Rad}(\bP_\lambda))\} \\ &= \; \bigsqcup_\lambda \;\ffl_{\bP_\lambda}^{t\gamma,\ell} \; , 
\end{align*} 
where the second equality follows from \cite[Lemma 4.3]{RW}.
\end{proof}

We finish this subsection by a brief discussion of the symmetries of $\Gr^{\gamma,\zeta}$. The $\Lambda\cong T(\cK)/T(\cO)$-action on $\Gr^\gamma$ coming from the stabilizer action commutes with $\zeta$, so we get a $\Lambda$-action on $\Gr^{\gamma,\zeta}$ and this preserves the decomposition in Proposition \ref{prop:affspalt}. On $H^*(\Gr^{\gamma,\zeta}),$ there is also an action of $W$ coming from varying $s\in \ft^{reg}$. These assemble to a left action of $\widetilde{W}$ on the cohomology. 
There is also another commuting right action of $\widetilde{W}$ on $H^*(\Gr^{\gamma,\zeta})$ coming from the Springer action. As mentioned before, the left action is called the dot action; we will call the right action the Springer action or, following \cite{CM}, the star action. These actions will be denoted $\tW\cdot$ and $\tW*$, respectively.

\subsection{A dimension formula for the cohomology}
\label{sec:dimformula}
In this section, we compute $\dim H^*(\Gr^{\gamma,\zeta})$ using affine Springer theory. 
 The strategy is as follows. Using Proposition \ref{prop:affspalt} we are reduced to computing each $\dim H^*(\ffl_{\bP_\lambda}^{t\gamma,\ell})^{\tW\cdot}$ individually and summing over the blocks. The first part is done using Springer theory in Proposition \ref{prop:invariants}. For the second part, we enumerate the blocks with a given stabilizer using the permutation representation of $W$ on $\Lambda/\ell \Lambda$. Finally, these results are shown to combine in the desired way in Theorem \ref{thm:dimension}.

Let 
$$\mathbf{e}=\frac{1}{|W|}\sum_{w\in W}w, \; \mathbf{e}^-=\frac{1}{|W|}\sum_{w\in W}(-1)^{\ell(w)}w$$ be the symmetrizing and antisymmetrizing idempotents for $W$. If $W_\lambda\subseteq W$ is a parabolic subgroup, we also let $\mathbf{e}_\lambda,\mathbf{e}_\lambda^-$ denote the corresponding idempotents. 

\begin{proposition}
\label{prop:invariants}
As ungraded $W$-representations, for the Springer action of $W$,
$$H^*(\ffl_{\bP_\lambda}^{t\gamma,\ell})^{\widetilde{W}\cdot} \cong \C[\Lambda/(h+1)\Lambda]\mathbf{e}^-_\lambda.$$ 
Here $W_\lambda$ is the stabilizer of $\lambda\in \Lambda$ inside $W$.
\end{proposition}
\begin{proof}
By \cite[Theorem 1.2]{BAL} we have that $H^*(\ffl_{\bI}^{t\gamma,\ell})^{\widetilde{W}\cdot}$ is isomorphic to $$\C[\Lambda/(h+1)\Lambda]$$ as a $W-$representation. We note that this is the $\sgn$-twist of $\DR_W$ \cite{Gordon}.
On the other hand, we claim that $H^*(\ffl_{\bP_\lambda}^{t\gamma,\ell})^{\widetilde W\cdot} \cong  H^*(\ffl_{\bI}^{t\gamma,\ell})^{\widetilde{W}\cdot}\mathbf{e}_\lambda^-$.
Note that there is a natural inclusion $$\ffl_{\bP_\lambda}^{t\gamma,\ell}\to \Fl_{\bP_\lambda}^{t\gamma,\ell}:=\{g\in \Fl_{\bP_\lambda}|\Ad_{g^{-1}}t^\ell s\in\Lie(\bP_\lambda)\},$$ 
and that we always (i.e. for any parahoric containing $\bI$ and any regular semisimple $\gamma$) have a Cartesian diagram
\[\begin{tikzcd}
	{\mathcal{F}l_{\mathbf{I}}^\gamma} && {[\widetilde{\mathfrak{l}}_{\mathbf{P}}/L_{\mathbf{P}}]} \\
	\\
	{\mathcal{F}l_{\mathbf{P}}^\gamma} && {[\mathfrak{l}_{\mathbf{P}}/L_{\mathbf{P}}]}
	\arrow[from=1-1, to=1-3]
	\arrow[from=1-1, to=3-1]
	\arrow[from=1-3, to=3-3]
	\arrow[from=3-1, to=3-3]
\end{tikzcd} ,\]
where the right-hand column is the Grothendieck--Springer resolution for $L_\bP$, the Levi quotient of $\bP$.
Taking the fiber at $0$ of the bottom map gives exactly $\ffl_\bP^\gamma$. The cohomology of this fiber is exactly the $W_\lambda$-antisymmetric part of the pullback of the Springer sheaf (see \cite[Lemma 2.2]{GKO}), so after noting that everything commutes with the $\widetilde{W}$-action, we are done.
\end{proof}

Recall that $$\mathbf{a}=\{\lambda \in A|\forall \alpha\in \Phi^+, 0<\langle \lambda, \alpha\rangle<\ell\}$$ is the $\ell$-dilated fundamental alcove for $G$. We would like to compute the number of blocks $u^{\vee,\lambda}_\zeta$ for the small quantum group for $\lambda \in \overline{\mathbf{a}}\cap \Lambda$ of a given type.  
By the definiton of $\mathbf{a}$, there is a bijection 
$\overline{\mathbf{a}}\cap \Lambda\leftrightarrow \Lambda/\ell \Lambda$. 
The stabilizer in the finite Weyl group $W\subset \widetilde{W}$ of $q\in \Lambda/\ell \Lambda$ is by \cite[Proposition 4.1]{Sommers} a parabolic subgroup of $W$. Fixing the type alluded to above is exactly fixing the conjugacy class of the stabilizer $W_\lambda \subseteq W$. It follows from Proposition \ref{prop:invariants} that the contribution of $H^*(\ffl_{\bP_\lambda}^{\gamma,\ell})^{\tW\cdot}$ to the dimension only depends on this data.
\begin{remark}
Without our assumptions on $\ell$, which for example imply $\ell$ is ''very good" in the sense of \cite{Sommers}, the $W_\lambda$ appearing above are in general only so called {\em quasi-parabolic} subgroups of $W$. We will, however, not need them.
\end{remark}

By \cite[Theorem 7.4.2]{HaimanConjectures}, the total number of $W$-orbits in $\Lambda/\ell \Lambda$ is \begin{equation}\label{eq:allorbits}
\frac{1}{|W|}\prod (\ell+e_i),\end{equation}
and the number of regular orbits is 
\begin{equation}
\label{eq:regularorbits}
\frac{1}{|W|}\prod(\ell-e_i).\end{equation}
Here the numbers $e_i$ are  the exponents of $W$. They are defined by $e_i = d_i - 1$ where the $d_i$ are the degrees of the homogenous generators of $\C[\ft]^W$. The quantity in \eqref{eq:allorbits} merits a name and plays a significant role in so called Coxeter--Catalan combinatorics. See for example \cite{ThielThesis}. 
\begin{definition}
\label{def:coxetercatalan}
Let $m$ be coprime to $h$. We define the {\em $m/h$-Coxeter--Catalan number of $W$} to be
$${\rm Cat}_W(m,h)=\frac{1}{|W|}\prod (m+e_i).$$
\end{definition}
Going back to our enumeration problem, 
let $\{W_\lambda\}$ be a set of representatives of parabolic subgroups of $W$. Now, $\C[\Lambda/\ell \Lambda]$ is by definition a permutation representation of $W$, so by orbit-stabilizer splits as $$\C[\Lambda/\ell \Lambda]=\bigoplus_\lambda d_{\lambda,\ell} \Ind_{W_\lambda}^W (1)\;,$$ where $d_{\lambda,\ell}\in \Z_{\geq 0}$. 
On the other hand, by Proposition \ref{prop:affspalt} and Proposition \ref{prop:invariants}, 
we have that 
\begin{equation}\label{dhl}
\dim H^*(\Gr^{\gamma,\zeta})^{\tW\cdot}=\sum_{\lambda}d_{\lambda,\ell}\dim \C[\Lambda/(h+1)\Lambda]\mathbf{e}_\lambda^-.
\end{equation}

\begin{definition}
The {\em $\ell/h$-rational Kreweras number of type $\lambda$ for $G$} is by definition the coefficient $d_{\lambda,\ell}$ in the above decomposition. In other words, it is the number of $W$-orbits in $\Lambda/\ell \Lambda$ with a given stabilizer.
\end{definition}
Explicit formulas for the Kreweras numbers for classical groups can be found in \cite{Sommers}; in type A they are simply multinomial coefficients by \cite{ALW}. There is also a general formula in terms of hyperplane arrangements \cite[Proposition 5.1]{Sommers}. We note that since each $\Ind_{W_\lambda}^W(1)$ contributes a trivial representation of $W$, the sum of $d_{\lambda,\ell}$ equals the dimension of $\C[\Lambda/\ell \Lambda]^W$. On the other hand, we have the following observation, originally going back to Haiman's work.
\begin{proposition}
\label{prop:kreweras}
The sum of the Kreweras numbers over the representatives $\{W_\lambda\}$ is the $\ell/h$-Coxeter--Catalan number for $W$:
$$\sum_\lambda d_{\lambda,\ell}={\rm Cat}_W(\ell,h).$$ In particular, $\dim \C[\Lambda/\ell \Lambda]={\rm Cat}_W(\ell,h)$.
\end{proposition}

The main result of this section is the following theorem, which also proves Theorem \ref{thm:mainthm} from the introduction.
\begin{theorem}
\label{thm:dimension}
Assume $\ell$ is as in the introduction. Then
we have $$\dim H^*(\Gr^{\gamma,\zeta})^{\tW\cdot}={\rm Cat}_W(\ell(h+1)-h,h).$$
\end{theorem}
\begin{proof}
We may interpret the summation over $\lambda$ on the RHS of Eq. \eqref{dhl} as follows.
Each orbit of type $\lambda$ contributes an $\Ind_{W_\lambda}^W (1)$ to the representation $\Lambda/\ell \Lambda$. On the other hand, $\dim \C[\Lambda/(h+1)\Lambda]\mathbf{e}_\lambda^-$ is by Frobenius reciprocity
$$\dim\Hom_W(\Ind_{W_\lambda}^W(1),\C[\Lambda/(h+1)\Lambda]\otimes \sgn),$$ so we can write $$\dim H^*(\Gr^{\gamma,\zeta})^{\tW\cdot}=\dim\Hom_W(\C[\Lambda/\ell \Lambda],\C[\Lambda/(h+1)\Lambda]\otimes\sgn).$$
By self-duality of $W-$representations and adjunction, the latter is
$$=\dim\Hom_W(\sgn, \C[\Lambda/\ell \Lambda] \otimes \C[\Lambda/(h + 1)\Lambda]).$$   Finally, by our assumptions on $\ell$, the map $$\Lambda/\ell \Lambda\cong (\Z/\ell(h+1)\Z)^r\to(\Z/\ell\Z)^r\times(\Z/(h+1)\Z)^r\cong \Lambda/\ell \Lambda\times \Lambda/(h+1)\Lambda$$ given by
$$(a_1,\ldots,a_r)\mapsto ((a_1 \text{ mod } \ell,\ldots,a_r\text{ mod }\ell),(a_1\text{ mod }(h+1),\ldots,a_r\text{ mod }(h+1))$$ is an isomorphism of $W-$sets.
Therefore, $\C[\Lambda/\ell \Lambda]\otimes \C[\Lambda/(h+1)\Lambda]\cong \C[\Lambda/\ell(h+1)\Lambda]$ as $W$-representations. 
 Finally, we have that
\begin{equation}\label{eq:chineseremainder}
H^*(\Gr^{\gamma,\zeta})^{\tW\cdot}=\dim\Hom_W(\sgn, \C[\Lambda/\ell(h + 1)\Lambda]).\end{equation}
To conclude the proof, we need the following lemma.
\begin{lemma}
\label{lem:signtwist}
Let $m>h$ be coprime to $h$. Then we have an isomorphism of vector spaces
$$\C[\Lambda/m \Lambda]\mathbf{e}^-\cong  \C[\Lambda/(m-h)\Lambda]\mathbf{e} \;.$$
\end{lemma}
\begin{proof}
Since only the regular orbits contribute sign representations, this follows from the relation between the number of all orbits vs. the number of regular orbits of $W$ in $\Lambda/m\Lambda$, 
given by Eq. \eqref{eq:allorbits} and Eq. \eqref{eq:regularorbits}.
\end{proof}
To finish the proof of Theorem \ref{thm:dimension}, the RHS of Eq. \eqref{eq:chineseremainder} is by Lemma \ref{lem:signtwist} $$\C[\Lambda/\ell(h+1)\Lambda]\mathbf{e}^-\cong \C[\Lambda/(\ell(h+1)-h)\Lambda]^W.$$ The dimension of this space is ${\rm Cat}_W((h+1)\ell-h,h)$ by Proposition \ref{prop:kreweras}.
\end{proof}
\begin{remark}
We believe that the $\gcd(h+1,\ell)=1$ is not needed.
\end{remark}
We now give some examples of Theorem \ref{thm:dimension}.

\begin{example}
\label{ex:dihedral}
For $W$ of dihedral type, there are three distinct types of orbits, which one can compute by hand or using Eqs. \eqref{eq:allorbits}--\eqref{eq:regularorbits}, generalizing the Alfano--Reiner results from \cite[Section 7.5]{HaimanConjectures}. For example for $G$ of type $B_2$, we have $1$ maximally singular orbit (the origin), $\frac{(\ell-1)(\ell-3)}{8}$ regular orbits, and $\ell-1$ subregular orbits.
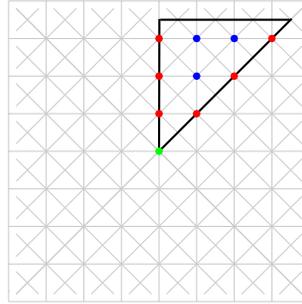
\begin{figure}[h]
\begin{center}
\begin {tikzpicture}
\begin {rootSystem}{B}
\draw[thick] \weight{0}{-0.4} -- \weight{0}{7.45};
\draw[thick] \weight{0.2}{-0.4} -- \weight{-3.79}{7.6};
\draw[thick] \weight{0.25}{7} -- \weight{-3.75}{7};
\wt [green]{0}{0}
\wt [red]{0}{2}
\wt [red]{0}{4}
\wt [red]{0}{6}
\wt [red]{-1}{2}
\wt [blue]{-1}{4}
\wt [red]{-2}{4}
\wt [blue]{-1}{6}
\wt [blue]{-2}{6}
\wt [red]{-3}{6}
\end {rootSystem}
\end {tikzpicture}
\caption{The dilated fundamental alcove for $B_2$ and dominant weights in it for $\ell=7$. Different colors correspond to different stabilizer types.}
\end{center}
\end{figure}
For $G$ of type $G_2$, we have the origin, $\frac{(\ell-1)(\ell-5)}{12}$ regular orbits, and $\ell-1$ subregular orbits.

From these numbers, we get the following dimension formulas. In the case of $B_2$  one computes 
$$\dim H^*(\Gr^{\gamma,\zeta})^{\tW\cdot}=25\cdot \frac{(\ell-1)(\ell-3)}{8}+10\cdot (\ell-1)+1,$$ and in the case of $G_2$ one has 
$$49\cdot \frac{(\ell-1)(\ell-5)}{12}+21\cdot (\ell-1)+1.$$
\end{example}
\begin{remark}
We remark that the above examples can be checked to match the Hochschild cohomology computations in \cite{HemelsoetVoorhaar}, where the dimension of the center of the quantum group is computed in several examples using coherent techniques. In particular, in \cite[Sections 4 and 5]{HemelsoetVoorhaar} the dimension of the $G^\vee$-invariant part of the center of the small quantum group is computed for all blocks in types $A_1, A_2, A_3, A_4, B_2, G_2$. These dimensions match the result of Theorem \ref{thm:dimension}, so Conjecture \ref{BBASV-conj} is confirmed in all these cases. 
\end{remark}

\section{Diagonal coinvariants}
\label{sec:dr}
Recall that the Cartan subalgebra $\ft \subset \fg$ carries an irreducible representation of the Weyl group $W$. Consider the ring of diagonal coinvariants $${\rm DR}_W:=\C[\ft \times \ft^*]/\C[\ft \times \ft^*]_+^W.$$ 
Here, $\C[\ft\times \ft^*]$ is naturally bigraded by giving a basis of $\ft^*$ bidegree $(1,0)$ and the dual basis in $\ft$ bidegree $(0,1)$. By definition, $\C[\ft \times \ft^*]_+^W$ is the doubly homogeneous ideal of diagonally invariant polynomials without constant term for $W$. In \cite{Gordon} a further representation-theoretically significant quotient ${\rm DR}_W\twoheadrightarrow\DR_W$ is defined and its structure as a $W$-module is studied.

In particular, it is shown in {\em loc. cit.} that the dimension of $\DR_W$ is $(h+1)^r$, where $r = {\rm rank}(\fg)$, and $h$ its Coxeter number \cite[Theorem 1.4]{Gordon}. Moreover, as a $W$-representation, $\DR_W\cong \sgn\otimes \C[\Lambda/(h+1)\Lambda]$.
Let $\lambda\in \Lambda,$ and $W_\lambda\subseteq W$ be the stabilizer of $\lambda$ as before. Consider the space of invariants \begin{equation}\label{eq:DRW}\DR_W^\lambda:=\DR_W^{W_\lambda}=\Hom_{W_\lambda}(\triv,\DR_W).\end{equation} 

By Frobenius reciprocity, the latter is the same as $\Hom_{W}(\Ind_{W_\lambda}^W(\triv),\DR_W)$. By Proposition \ref{prop:invariants}, 
\begin{equation} \label{block:DR}
H^*(\ffl_{\bP_\lambda}^\gamma)^{\tW\cdot}\cong \C[\Lambda/(h+1)\Lambda]\mathbf{e}_\lambda^-\cong (\DR_W \otimes \sgn)\mathbf{e}_\lambda^-  \;.
\end{equation} 
Endowing $\DR_W\otimes\sgn$ with its natural  $W-$invariant bigrading we may try to upgrade results from the previous section to include this bigrading.

\subsection{Rational shuffle theorems and the center in type A}
In this subsection, we will reinterpret the proof of Theorem \ref{thm:dimension} for $G$ of type $A$ using the language of rational shuffle theorems \cite{CMShuffle, Mellit}. We hope this will give an illuminating inroad to understanding the bigrading on the center.

When $\fg = \mathfrak{gl}_n$ we have $W = S_n$ and we write 
$$\DR_{S_n}=\DR_{n} = \mathrm{DR}_n = \frac{\C[x_1,\ldots,x_n,y_1,\ldots,y_n]}{\C[x_1,\ldots,x_n,y_1,\ldots,y_n]_+^{W}}.$$  Similarly, if $S_\lambda\subseteq S_n$ is a Young subgroup, we write $\DR_n^\lambda$ for the space of invariants in Eq. \eqref{eq:DRW}.

Let $\Sym_{q,t}[X]$ be the ring of symmetric functions over $\Q(q,t)$ in the alphabet $X=\{x_1,x_2,\ldots\}$ and let $\nabla$ be the nabla operator of \cite{BGHT99}, diagonal in the basis of modified Macdonald polynomials. 
Let $\{e_\lambda\},\{p_\lambda\},\{h_\lambda\},\{m_\lambda\},\{s_\lambda\}$ be the bases of elementary, power sum, complete homogeneous, monomial, and Schur symmetric functions, and $\omega=\omega_X$ the usual involution on symmetric functions.

Consider the Frobenius characteristic map $$\Frob_{q,t}: \Rep_{\Z^2-\text{graded}}(W)\to K_0(\Rep_{\Z^2-\text{graded}}(W)),$$ which takes a doubly graded representation to its class in the Grothendieck group, where a representation in bigrading $(i,j)$ is weighted by $q^it^j$. 
When $W=S_n$, we can and will further identify $K_0(\Rep_{\Z^2-\text{graded}}(W))\cong \Sym_{q,t}^n$ by sending the Specht module labeled by $\lambda$ to the Schur function $s_\lambda$. 

Analogously to Eq. \eqref{eq:DRW}, the bigraded dimension of $\DR^\lambda_n$ is given by the Hall inner product of Frobenius characters:  \begin{equation} \label{DRlambda} 
{\rm dim}_{q,t}(\DR_W^\lambda) =  \langle \Frob_{q,t}(\Ind_{W_\lambda}^W(\triv)),\Frob_{q,t}(\DR_W)\rangle.
\end{equation} Obviously, $\dim(\DR_W^\lambda)=\dim_{q,t}(\DR_W^\lambda)_{q=t=1}$.

The $q,t$-Frobenius character of $\DR_W$ is shown in \cite{HaimanVanishing} to be $\Frob_{q,t}(\DR_W)=\nabla e_n$.
In this case, we also have the following more explicit statement about the bigraded dimension of $\DR_n^{W_\lambda}$.
\begin{proposition}
\label{prop:bigradedinvariants}
Let $\lambda \in \Lambda $ and $W_\lambda \subset S_n$ be the stabilizer of $\lambda$. Then 
\[ \dim_{q,t}(\DR_n^\lambda) = \langle h_\lambda, \nabla e_n \rangle ,\] 
\end{proposition} 
\begin{proof} 
It is well known that the ungraded Frobenius character of $\Ind_{W_\lambda}^W(triv)$ is given by $h_\lambda$, the homogeneous symmetric function attached to $\lambda$. 
Therefore, we compute using (\ref{DRlambda}): 
\[ \dim_{q,t}(\DR^\lambda_n)= \langle \Frob_{q,t}(\Ind_{W_\lambda}^W(\triv)),\Frob_{q,t}(\DR_n)\rangle = \langle h_\lambda, \nabla e_n\rangle .\] 
\end{proof} 
We remark that the expression $\nabla e_n$ can also be written as $P_{n+1,n}\cdot 1$, where $P_{m,n}$ for $m,n\geq 0$ are certain elliptic Hall algebra operators as in \cite{Negut}, acting on the space $\Sym_{q,t}$. These operators play an important role for example in rational Catalan combinatorics \cite{CMShuffle, ALW} and the point-counting on affine Springer fibers \cite{KiTs}. Combinatorial expressions for $P_{m,n}\cdot 1$ are the subject of the rational shuffle theorem of e.g. \cite{Mellit}.
We refer to \cite{Negut} for the precise definition of the operators $P_{m,n}$ and their action on $\Sym_{q,t}$.
Unsurprisingly, these operators  also provide a convenient language to understand the appearance of the rational Catalan numbers ${\rm Cat}_n((n+1)\ell-n,n)$ from Theorem \ref{thm:dimension}, which we now reprove in this language.
\begin{theorem}
\label{thm:catalan}
Suppose $\ell$ is as in the introduction, i.e. odd and $n\not\equiv 0, -1$ mod $\ell$. We have $$\dim H^*(\Gr^{\gamma,\zeta})^{\tW\cdot}={\rm Cat}_n((n+1)\ell-n,n)=\frac{1}{(n+1)\ell}\binom{(n+1)\ell}{n},$$ the rational $((n+1)\ell-n,n)$-Catalan number.
\end{theorem}
\begin{proof}
By  \cite[Proposition 1.7]{BEG} and \cite[§5.2]{GN}, we can write the rational Kreweras numbers as 
\begin{equation}
\label{eq:krewerasnumbers}
d_{\lambda,\ell}=\langle P_{\ell,n}\cdot 1, m_\lambda\rangle|_{q=t=1}
\end{equation}

where $m_\lambda$ are the monomial symmetric functions and $P_{m,n}$ for $m,n\geq 0$ are the EHA operators from above.
By Eq. \eqref{dhl}, we have
\begin{equation}
\label{eq:dimensionsum}
\dim H^*(\Gr^{\gamma,\zeta})=\sum_{\lambda \vdash n}d_{\lambda,\ell}\langle h_\lambda, \nabla e_n\rangle|_{q=t=1}.
\end{equation}
By combining equations \eqref{eq:krewerasnumbers} and \eqref{eq:dimensionsum}, we can use linearity of the scalar product to get 
\begin{equation}
\label{eq:pairing}
\langle \sum_{\lambda} h_\lambda\langle P_{\ell,n}\cdot 1, m_\lambda\rangle, \nabla e_n\rangle|_{q=t=1}=
\langle \omega P_{\ell,n}\cdot 1, \nabla e_n\rangle|_{q=t=1} \;.\end{equation}
Write now $\mathbf{p}_{\ell/n}=(P_{\ell,n}\cdot 1)|_{q=t=1}$. Again by \cite[Proposition 1.7]{BEG} and \cite[§5.2]{GN}, this is simply the ungraded Frobenius characteristic of $\C[\Lambda/\ell \Lambda]$. In particular, we have the specialization $\mathbf{p}_{\ell/n}(1)={\rm Cat}_n(n,\ell)$.

By standard properties of the Hall inner product, the RHS of Eq. \eqref{eq:pairing} equals $(\mathbf{p}_{\ell/n}\star (\nabla e_n)|_{q=t=1})(1)$ where $\star$ denotes the Kronecker product on symmetric functions.
Just as in the proof of Theorem \ref{thm:dimension}, when $n\not\equiv 0, -1$ mod $\ell$, the Chinese remainder theorem implies that $$\C[\Lambda/\ell \Lambda\times \Lambda/(n+1)\Lambda]\cong \C[\Lambda/(\ell(n+1)) \Lambda]$$ as $S_n$-representations. In particular, by the above we have
$$\mathbf{p}_{\ell/n}\star (\nabla e_n)|_{q=t=1}=\mathbf{p}_{(\ell(n+1))/n}.$$ 
Since $\mathbf{p}_{(\ell(n+1))/n}(1)={\rm Cat}_n((n+1)\ell-n,n)$, we are done.
\end{proof}

\subsection{The bigrading and the subregular block in type A}
In this subsection, we discuss the bigrading on the side of the quantum group in the case of a subregular block and its relation to the shuffle theorem. We then return to affine Springer fibers and upgrade Theorem \ref{thm:dimension} to include the natural bigrading for $G=SL_n$, assuming the main conjecture from \cite{CM}. 
\subsubsection{The subregular block}
In the notations of Section \ref{sec:dimformula}, when $\lambda\in \Lambda$ has stabilizer $W_\lambda\cong S_2$, we call $\lambda$ a subregular weight. In type A, this implies that the associated partition is $\lambda=(n-1,1)$.
In this case, Proposition \ref{prop:bigradedinvariants} has the following explicit corollary.
\begin{corollary} \label{subreg} 
When $W = S_n$ and $\lambda=(n-1,1)$, we have $$\dim_{q,t}(\DR^\lambda_n)=\langle h_\lambda, \nabla e_n\rangle=\sum_{k=0}^{n-1} \sum_{i+j=k} q^it^j.$$ 
\end{corollary}
\begin{proof}
The Shuffle Theorem of \cite[(1.1)]{CMShuffle} states that $$\nabla e_n=\sum_{\pi\in PF_n} q^{\text{area}(\pi)}t^{\text{dinv}(\pi)}x_\pi \;,$$ where $\pi\in PF_n$ is a parking function on $n$ letters, and area and dinv are certain combinatorial statistics (see \cite{CMShuffle} for the definition). The monomial $x_\pi$ is a monomial in the alphabet $\{x_1,\ldots,x_n,\ldots\}$ associated to $\pi$.
Collecting all the monomials in the $S_n$-orbit of a fixed $\pi$ and using the orthogonality of the bases $\{m_\lambda\},\{h_\mu\}$ for the Hall inner product, we see that $\langle h_\lambda, \nabla e_n\rangle$ is a weighted count of Dyck paths whose associated monomial is $\lambda$. For $\lambda=(n-1,1)$, these are Dyck paths differing from the one with minimal area by allowing an extra horizontal step (compare \cite{EHHK}, where a similar result is proved using Schr\"oder paths). Fixing the length of this step, we get $n-\text{length}$ Dyck paths, each of which has the same area. It is easy to see that they all have a different dinv statistic. In total, we get $\binom{n+1}{2}$ Dyck paths, each with distinct statistics. This completes the proof.
\end{proof}
This corollary has an immediate application to the small quantum group, which does not involve the geometry of affine Springer fibers. This resolves a conjecture of the third author and Qi \cite[Conjecture 4.9(3)]{LQ1} at the level of bigraded vector spaces.
\begin{corollary} 
Let $\fg^\vee = {\mathfrak{sl}}_n$ and let $Z(\fu_\zeta^{\vee,\lambda})$ denote the block of the center of the small quantum group $\fu_\zeta(\fg^\vee)$ with $\lambda$ a subregular weight. Let $P_\lambda \subset G= SL_n$ be the parabolic subgroup associated to $\lambda$ and $\widetilde{N_\lambda} \cong T^*(G/P_\lambda)$ the Springer resolution. The additional grading of the coherent sheaf of poly-vectorfields $\wedge^j T \widetilde{N_\lambda}$ is given by the induced action of $\C^*$ along the fibers of the Springer resolution.  Then there are isomorphisms of bigraded vector spaces  
\[ Z(\fu_\zeta^{\vee,\lambda})^{i,j} \cong H^i (\widetilde{N_\lambda}, \wedge^j T \widetilde{N_\lambda})^{-i-j}  \cong\left(\DR_n^\lambda \right)^{\binom{n+1}{2} - \frac{i+j}{2}, \frac{j-i}{2}} . \]
\end{corollary} 
\begin{proof} 
The first isomorphism is a particular case of Theorem \ref{block:coherent} which is a consequence of Theorem 7 in \cite{BL}. The bigraded dimensions of the equivariant coherent sheaf cohomologies in the case where $\lambda$ is subregular and $G/P_\lambda \cong{\mathbb P}^k$ are computed in Theorem 3.3 of \cite{LQ2}. They match exactly the bigraded dimensions of $\DR_n^\lambda$ obtained in Corollary \ref{subreg}. 
\end{proof} 
This shows in particular that for $G$ of type $A$ and a singular block $Z(\fu_\zeta^{\vee, \lambda})$ such that $G/P_\lambda$ is a projective space, the cohomology of the corresponding affine Springer fiber is isomorphic to the whole singular block of the center.  
Note that in this case, the whole block of the center of the small quantum group is $G^\vee$-invariant.

\subsubsection{The positive part of the affine Springer fiber and diagonal coinvariants}
We now define a bigrading on $H^*(\Gr^{\gamma,\zeta})^{\tW \cdot}$ by applying results from \cite{CM}. We focus on $H^*(\Fl_\bI^\gamma)^{\tW \cdot}$, in other words, the principal block. On the principal block, the obtained bigrading conjecturally coincides with that of $\DR_n$.
Below, we will use symmetric functions in two sets of variables $X, Y$. We also use standard plethystic notation,  such as writing $f[XY]$ for the result of substituting $p_k(X)$ by $p_k(X)p_k(Y)$ in the expansion of $f\in \Sym_{q,t}$ in the basis of the $p_\lambda$. See the book \cite{Haglund} for details.

For the rest of this section, suppose $\gamma=ts$ for $s\in \ft^{reg}$ as before. For technical reasons, we also take $G=GL_n$. It is not simply connected but all the definitions concerning affine Springer theory apply in this case as well. The relation to the $SL_n$ case and $\fu_\zeta^{\vee}$ is explained below. Let $\Gr_G^+$ be the positive part of the affine Grassmannian, which on $\C$-points consists of lattices $\Lambda\subset \cK^n$ contained in the standard lattice $\cO^n$. The positive part of the affine flag variety $\Fl_\bI^+$ is defined as the preimage of $\Gr_G^+$ under the natural projection $\Fl\to\Gr$. 

The positive part of the affine Springer fiber $\Fl^\gamma_\bI$ is similarly defined to be $$\Fl^{\gamma,+}_\bI:=\Fl_\bI^+\cap \Fl^{\gamma}_\bI \;.$$ Its equivariant Borel--Moore homology
$H_*^T(\Fl^{\gamma,+}_\bI)$ is naturally bigraded by the connected component $t^i\in \pi_0(\Fl_\bI^+)=\Z$ and the (half of the) cohomological grading $q^j\in \Z$. Moreover, this space carries two bigraded $S_n$-actions, one from the Springer action and one from the monodromy as $s$ moves in a family. We will still call these the Springer or star action and the equivariant centralizer-monodromy or dot action, respectively. 

These positive parts were studied in \cite{CM,CM1,Kiv}, and they are closely related to the isospectral Hilbert scheme of $\C^2$ as well as Haiman's work \cite{HaimanVanishing}.
The main result we will need is the following theorem, which is the $k=1$ specialization of \cite[Theorem A]{CM}. 

\begin{theorem}
The Frobenius character of $H_*^T(\Fl^{\gamma,+}_\bI)$ for the $S_n \times S_n$-action is given by 
$$\Frob_{q,t,X,Y}(H_*^T(\Fl^{\gamma,+}_\bI))=
q^{-\binom{n}{2}}\omega_X\nabla e_n\left[\frac{XY}{(1-q)(1-t)}\right].$$
\end{theorem}
\begin{remark}
Compare this also to \cite[Conjecture 3.7]{CM1}, proved in \cite[Remark 7.3]{BAL}.
\end{remark}
\begin{corollary}
$$\Frob_{q,t,X,Y}(H_*(\Fl^{\gamma,+}_\bI))=q^{-\binom{n}{2}}\omega_X\nabla e_n\left[\frac{XY}{1-t}\right].$$
\end{corollary}
\begin{proof}
Since $\Fl^{\gamma,+}_\bI$ is equivariantly formal, the generators of $H_*^T(pt)$ form a regular sequence in $H_*^T(\Fl^{\gamma,+}_\bI)$.
Now apply \cite[Lemma 3.6]{HaimanVanishing}.
\end{proof}

Next, note that the positive part of the lattice $\Lambda$, i.e. $\Lambda^+ := \Z_{\geq 0}^n$, acts on $\Fl^{\gamma,+}_\bI$. As explained in \cite{KiTr}, we have 
$$\Fl^{\gamma,+}_\bI/\Lambda^+\cong \Fl^{\gamma}_\bI/\Lambda.$$ Note that the RHS also equals the similar lattice quotient for $G=SL_n$. Further, from for example the explicit description as the module called ''$M$" in \cite{CM}, we see $$H_*^T(\Fl^{\gamma}_\bI)\cong H_*^T(\Fl^{\gamma,+}_\bI)\otimes_{\C[\Lambda^+]}\C[\Lambda]$$ as $\C[\Lambda]$-modules.

Using the degeneration of the Cartan--Leray spectral sequence for the $\Lambda^+$ and $\Lambda$-actions on $\Fl^{\gamma,+}_\bI$, resp. $\Fl^{\gamma}_\bI$, we have 
\begin{lemma}
$$H_*(\Fl^{\gamma}_\bI/\Lambda)=\bigoplus_{i} \Tor_i^{\C[\Lambda^+]}(H_*(\Fl^{\gamma,+}_\bI),\C) \;.$$
\end{lemma}
Suppose we wanted to kill the lattice action instead of passing to the non-equivariant limit. Indeed, since $(H_*(\Fl^{\gamma,+}_\bI)_{\Lambda^+})^*\cong H^*(\Fl^{\gamma,+}_\bI)^\Lambda$, the bigraded Frobenius characters under $S_n$ are the same. Note that the coinvariant space is by definition $H_*(\Fl^{\gamma,+}_\bI)_{\Lambda^+}=\Tor_0^{\C[\Lambda^+]}(H_*(\Fl^{\gamma,+}_\bI),\C)$, so inherits a second grading from $H_*(\Fl^{\gamma,+}_\bI)$.

Again by \cite[Lemma  3.6]{HaimanVanishing}
$$\Frob_{q,t,X,Y}(H_*^T(U))=\omega_X\nabla e_n\left[\frac{XY}{1-q}\right],$$
where the LHS is the bigraded Frobenius characteristic equivariant Borel--Moore homology of a certain open fundamental domain $U$ of the lattice action defined in \cite[Definition 6.9]{CM}. It has only even dimensional nontrivial cohomology groups, as is implied from the formula. But interestingly, this does not mean that it is equivariantly formal, and indeed this space will have nontrivial odd usual Borel--Moore homology groups.

Finally, Eq. (4) of \cite{CM} implies that
$$\sum_{i\geq 0} (-1)^i \Frob_{q,t,X,Y}(\Tor_i^{\C[\Lambda^+]}(H_*(U),\C))=\omega_X \nabla e_n\left[XY\right].$$ 
The main conjecture of {\em loc. cit.} states that the $\Tor_i$ groups that appear on the left contain only those nontrivial representations $\chi_\lambda$ of the left $S_n$-action (the dot action) for $i=\iota(\lambda')$, $\iota$ being a certain combinatorial statistic from the nabla positivity conjecture. In particular, taking lattice invariants is the result of substituting $p_k(Y)=1$ in $e_n[XY]$, in other words taking the trivial component of the representation of the ''dot" action. This is the same as the $\Tor_0$ part, and so by \cite[Conjecture A]{CM} corresponds to tensoring out both $\mathbf{x}$ and $\mathbf{y}$ from $H_*^T(\Fl^{\gamma,+}_\bI)$ over $\C[\Lambda^+]\otimes \C[\ft]\cong \C[\mathbf{x},\mathbf{y}]$, without including higher derived functors.

Combining the above remarks, we have
\begin{theorem} \label{thm:bgrd}
Suppose \cite[Conjecture A]{CM} is true. 
Then $$\Frob_{q,t}(H^*(\Fl^{\gamma}_\bI)^\Lambda)=\omega_X \nabla e_n.$$ In other words, the bigraded structure of the $\Lambda$-invariants, as an $S_n$-representation, coincides with the sign-twist of the diagonal coinvariants.
\end{theorem}
Similarly, we may obtain the bigraded dimensions of $H^*(\ffl^{t\gamma_\ell,\ell}_{\bP_\lambda})^\Lambda$ in this way, where we use the notation $\gamma_\ell=t^\ell s$ to distinguish from the temporary assumption we made on $\gamma$ in the beginning of this subsection.
\begin{corollary}
\label{cor:bigradedsingularblocks}
For any $\lambda\in \Lambda/\widetilde{W}$, we have 
$$\dim_{q,t}(H^*(\ffl_{\bP_\lambda}^{t\gamma_\ell,\ell})^\Lambda)=\langle \Frob_{q,t}(H^*(\Fl_{\bI}^\gamma)^\Lambda),e_\lambda\rangle=\langle \omega \nabla e_n,\omega h_\lambda\rangle=\langle \nabla e_n, e_\lambda\rangle.$$
\end{corollary}

\section{The small quantum group}
\label{sec:smallquantum}
In this section, we discuss further implications of Theorem \ref{BBASV-thm} and the results of the previous two sections to the structure of $Z(u_\zeta^\vee)^{G^\vee}$.
We first cite a description for the center of the small quantum group based on 
the derived equivalence of categories between a certain category of representations of quantum groups at roots of unity and a derived category of $G \times {\mathbb C}^*$ equivariant coherent sheaves over the Springer resolution (see \cite{ABG}, \cite{BL}). 
For $\lambda \in \Lambda$, let $\fu_\zeta^{\vee,\lambda}$ denote the block corresponding to the $\widetilde{W}$-orbit of $\lambda$ as before. 
The following result is shown in \cite{LQ2}: 
\begin{theorem} \label{block:coherent}
Let $P^\vee_\lambda \subset G^\vee$ be the parabolic subgroup associated to $\lambda$ and $\widetilde{N_\lambda} \cong T^*(G^\vee/P^\vee_\lambda)$ the corresponding partial Springer resolution. There is an isomorphism of algebras  
\[ Z(\fu_\zeta^{\vee,\lambda}) \cong\bigoplus_{i+j+k=0}  H^i (\widetilde{N_\lambda}, \wedge^j T \widetilde{N_\lambda})^{k}  .  \]
Here $\wedge^\bullet T \widetilde{N_\lambda}$ is the coherent sheaf of poly-vectorfields, which has an induced action of $\C^*$ coming from dilating the fibers of the Springer resolution, and the degree $k$ tracks the grading induced by this action. 
 \end{theorem} 
We will use this result  to discuss some interesting previously known components of the center of the small quantum group and identify them in the framework of the affine Grassmannian model for the center. For later use, we denote $\tau \in H^0(\widetilde N_{\lambda}, \wedge^2 T\widetilde N_{\lambda})^{-2}$ to be the canonical Poisson bivector field.

\subsection{The Harish-Chandra center, the Higman ideal and the Verlinde quotient} 
The center $Z(\fu^\vee_\zeta)$ contains several interesting subalgebras. In this subsection, we define some of these and give some indications of their geometric meaning.
We define two commutative subalgebras of the dual small quantum group $(\fu_\zeta^\vee)^*$ by 
\begin{align*} c &= \{ f \in (\fu_\zeta^\vee)^* \; | \; f(ab) = f(ba) \;\; \forall a, b \in \fu_\zeta^\vee\} , \\ 
 c_l &= \{ f \in (\fu_\zeta^\vee)^* \; | \; f(ab) = f(b S^2(a)) \;\; \forall a, b \in \fu_\zeta^\vee\} , 
 \end{align*}
where $S$ is the antipode. 

For any $\fu_\zeta^\vee$-module $M$, its character defined as $\chi_M(x) = {\rm tr}_M(x)$ for $x \in \fu_\zeta^\vee$ is an element in $c$. The complexification of the Grothendieck ring  ${\rm K}_0(\fu_\zeta^\vee$-mod) is a subalgebra in $c$: 
\[ r(\fu_\zeta^\vee) = {\rm K}_0(\fu_\zeta^\vee {\text{-}} {\rm{mod}})\otimes_\Z \C \subset c .\]

The characters of the projective $\fu_\zeta^\vee$-modules form an ideal in $r(\fu_\zeta^\vee)$ with respect to the multiplication induced from the tensor product of the representations. 
We define the ideal 
\[  p(\fu_\zeta^\vee) = \{ \chi_P \; | \; P \; {\rm is} \; {\rm projective} \} \subset r(\fu_\zeta^\vee). \]

Following \cite{AP} we can define a larger ideal spanned by the characters of {\it negligible} tilting modules, spanned by the characters of those direct summands of tensor products of simple $\fu_\zeta^\vee$-modules with highest weight in the closure of the first dominant alcove $\overline{\mathbf{a}}$ that have zero quantum dimension: 
\[ q(\fu_\zeta^\vee) = \{ \chi_N \; | \; {\rm tr}_N(K_{2 \rho}) =0 \} \subset r(\fu_\zeta^\vee) ,\] 
where $\rho$ is half the sum of the positive roots of $\fg^\vee$. It is known that 
\[ p(\fu_\zeta^\vee) \subset q(\fu_\zeta^\vee) \subset r(\fu_\zeta^\vee) ,\] 
where the first inclusion follows from \cite[Propositions 3.5 and 5.8]{And92}. The quotient 
\[ v(\fu_\zeta^\vee) = r(\fu_\zeta^\vee)/q(\fu_\zeta^\vee) \]
is a semisimple commutative algebra, see  \cite{And92}, Section 4. 

Recall that $\fu_\zeta^\vee$ is a quasitriangular Hopf algebra with the invertible element $R = \sum R_1 \otimes R_2 \in \fu_\zeta^\vee \otimes \fu_\zeta^\vee$ such that $u = \sum S(R_2) R_1$ is invertible and has the property $S^2(x) = u x u^{-1}$ for any $x \in \fu_\zeta^\vee$. We can define the isomorphism of commutative algebras 
\[ \mu_l : c(\fu_\zeta^\vee) \to c_l (\fu_\zeta^\vee), \quad  \mu_l(f) = f(u \;\cdot -) .\]
Then we set  
\[ r_l(\fu_\zeta^\vee) = \mu_l(r(\fu_\zeta^\vee)), \quad  q_l(\fu_\zeta^\vee) = \mu_l(q(\fu_\zeta^\vee)), \quad  p_l(\fu_\zeta^\vee) = \mu_l(p(\fu_\zeta^\vee)). \] 

Recall from \cite{Dri89} that the map 
\[ J : c_l(\fu_\zeta^\vee) \to Z(\fu_\zeta^\vee) , \quad  J:  f \to m(f \circ S^{-1} \otimes {\rm id}) (R_{21}R_{12}) \] 
 defines an isomorphism from the space of the left-shifted tracelike functionals $c_l$ to the center of $\fu_\zeta^\vee$. Here $m$ denotes the multiplication $\fu_\zeta^\vee \otimes \fu_\zeta^\vee \to \fu_\zeta^\vee$. Restricted to $r_l$, it is an injective algebra homomorphism. 

\begin{definition} 
The Harish-Chandra center $Z_{\rm HC}$ is the subalgebra in the center $Z(\fu_\zeta^\vee)$  defined by 
\[ Z_{\rm HC} = J (r_l(\fu_\zeta^\vee)) .\]
\end{definition} 
Since the Harish-Chandra center descends from the center of the divided powers quantum group, it lies in the $G^\vee$-invariant part of the center of $\fu_\zeta^\vee$ (see \cite{LQ3}). 
\begin{definition} 
The Higman ideal in the center $Z(\fu_\zeta^\vee)$ is defined by 
\[ Z_{\rm Hig} = J(p_l(\fu_\zeta^\vee)). \]
\end{definition} 

\begin{remark} 
A more conventional definition of the Higman ideal in the center of a finite dimensional Hopf algebra is the following. Recall that $\fu_\zeta^\vee$ is a unimodular Hopf algebra, meaning that it contains a two-sided integral $\nu \in Z(\fu_\zeta^\vee)$ that is unique up to rescaling and such that $\nu x = \varepsilon(x) \nu$ and $x \nu = \varepsilon(x ) \nu$ for any $x \in \fu_\zeta^\vee$, where $\varepsilon : \fu_\zeta^\vee \to \C$ is the counit.
The Hopf algebra $\fu_\zeta^\vee$ is a left module over itself with respect to the Hopf adjoint action ${\rm ad}h (x) = \sum h_1 x S(h_2)$ for any $h, x \in \fu_\zeta^\vee$. Then the Higman ideal is defined as 
\[ Z_{\rm Hig} = {\rm ad}\nu (\fu_\zeta^\vee) .\] 
The equivalence of the two definitions in the case of the small quantum group is shown in \cite{LQ3}. 
\end{remark}

By analogy with the Verlinde algebra arising from the fusion category coming from the representation theory of quantum groups at roots of unity, we also define the semisimple Verlinde quotient of the center. 
\begin{definition} 
Let 
\[ Z_{\rm neg} = J (q_l(\fu_\zeta^\vee) \]
be the ideal spanned by the images by $J$ of the characters of $\fu_\zeta^\vee$-modules with zero quantum dimension. 
The Verlinde quotient of the center $Z(\fu_\zeta^\vee)$ is defined by 
\[ {\rm Ver} = Z_{\rm HC} / Z_{\rm neg} .\]
\end{definition} 
We have 
\[ Z_{\rm Hig} \subset Z_{\rm neg} \subset Z_{\rm HC} \subset Z(\fu_\zeta^\vee)^{G^\vee} .\]

\begin{proposition} 
Suppose that $\ell$ is very good. We have 
\[ {\rm dim} \, Z_{\rm HC} = \ell^{{\rm rk}(\fg^\vee)},  \quad 
{\rm dim} \, Z_{\rm Hig} = {\rm Cat}_W(\ell, h), \quad {\rm dim} \, {\rm Ver} = {\rm Cat}_W(\ell-h, h) .\] 
\end{proposition} 
\begin{proof} 
The dimension of the Harish-Chandra center equals to the number of inequivalent simple $\fu_\zeta^\vee$-modules, which are classified by the elements of the $\ell$-restricted weight lattice by Lusztig's tensor product theorem \cite{Lus89}. The dimension of the Higman ideal equals to the total number of blocks of $\fu_\zeta^\vee$ as computed in \cite{LQ3}. The dimension of ${\rm Ver}$ equals to the number of regular blocks of $\fu_\zeta^\vee$, which follows from the characterization of the negligeable tilting modules given in \cite{AP}. The  number of total and regular blocks equals to the number of total and regular orbits of $W$ in the $\ell$-restricted weight lattice and is given  in equations  (\ref{eq:allorbits}) and (\ref{eq:regularorbits}) respectively, expressed in terms of the rational Catalan numbers. 
\end{proof} 
Next we will describe the block decomposition of these subspaces in the center. 

\begin{proposition} 
\begin{enumerate}[label=(\alph*)]
\item Let $C^{W_\lambda}_W = \C[\ft]^{W_\lambda}/\C[\ft]_+^W$ denote the (partial) coinvariant algebra associated to the stabilizer subgroup $W_\lambda \subset W$. Then there is an isomorphism of algebras
\[ Z_{\rm HC}^\lambda : = Z_{\rm HC} \cap Z(\fu_\zeta^{\vee, \lambda}) \cong C^{W_\lambda}_W .\]
In particular, $Z_{\rm HC}^\lambda$  has the dimension ${\rm dim} \, Z_{\rm HC}^\lambda = [W : W_\lambda]$.  
\item We have
\[ Z_{\rm Hig}^\lambda = Z_{\rm Hig} \cap Z(\fu_\zeta^{\vee, \lambda}) \cong{\rm AnnRad}(C^{W_\lambda}_W).\]
In particular, ${\rm dim}\, Z_{\rm Hig}^\lambda = 1$  for all $\lambda$. 
\item We have 
\[  {\rm Ver}^\lambda = {\rm Ver} \cap Z(\fu_\zeta^{\vee, \lambda}) \cong \left[ \begin{array}{ll} 
C^\lambda / {\rm Rad} \, C^{W_\lambda}_W, & \lambda \;\; {\rm regular} \\
0, & {\rm otherwise.}  \end{array} \right. \]
The dimension of ${\rm Ver}^\lambda$ is one if $\lambda$ is regular and zero otherwise. 
\end{enumerate} 
\end{proposition} 
\begin{proof} 
\begin{enumerate}[label=(\alph*)]
\item The algebraic structure of the Harish-Chandra center blocks was computed in \cite{BrownGordon} to be isomorphic to the coinvariant algebra $C^{W_\lambda}_W$, which carries a natural action of $W$ and has dimension $ [W : W_\lambda]$. 
\item Since $Z_{\rm Hig}$ annihilates the radical of $Z(\fu_\zeta^\vee)$, and $C^{W_\lambda}_W$ is a local Frobenius algebra for each $\lambda$, we have that $Z_{\rm Hig}^\lambda$ spans the one-dimensional annihilator of the radical of $C^{W_\lambda}_W$. 
In particular, ${\rm dim} \, Z_{\rm Hig}^\lambda = 1$.   
\item  Following the description of the characters of the negligible modules in \cite{AP}, we conclude that they span the radical (of codimension 1) of each regular block in the algebra $r_l(\fu_\zeta^\vee)$. All characters of modules in singular blocks come from the negligible modules. Using the algebra isomorphism $J : r_l(\fu_\zeta^\vee) \cong Z_{\rm HC}$ that maps $q_l(\fu_\zeta^\vee) \to Z_{\rm neg}$ allows us to conclude. 
\end{enumerate} 
\end{proof} 
We want to understand the place of $Z_{\rm HC}, Z_{\rm Hig}$ and ${\rm Ver}$ inside the new  model for the center of the small quantum group that comes from the isomorphism of algebras 
\[Z(u^\vee_\zeta)^{G^\vee}_{geom} \cong H^*(\Gr^{\gamma,\zeta})^{\widetilde{W} \cdot } \cong\bigoplus_{[W_\lambda]} d_{\lambda, \ell} H^*(\ffl_{\bP_\lambda}^{t\gamma,\ell})^{\widetilde W \cdot} . \] 
It was remarked in \cite{BL} that the Harish-Chandra center $Z_{\rm HC}^\lambda \subset Z(\fu_\zeta^{\vee, \lambda})$ can be identified in the framework of Theorem \ref{block:coherent} with the canonically defined subalgebra corresponding to the cohomology of the (partial) flag variety. 
Let $\widetilde{N_\lambda} \cong T^*(G^\vee/P^\vee_\lambda)$ and  $d = {\rm dim}_\C (G^\vee/P^\vee_\lambda)$. Then 
\[ Z_{\rm HC}^\lambda  \cong \oplus_{0 \leq i \leq d} H^i(\widetilde{N}_\lambda, \wedge^i T\widetilde{N}_\lambda)^{-2i} \cong H^{\bullet} (G^\vee/P^\vee_\lambda) \cong C^{W_\lambda}_W .  \] 
Since $Z_{\rm Hig}^\lambda$ and ${\rm Ver}^\lambda$ are respectively the socle and the head of the ring $C_W^{W_\lambda}$, we have (see also \cite{LQ3}): 
\[ Z_{\rm Hig}^\lambda \cong H^d(\widetilde{N}_\lambda, \wedge^d T\widetilde{N}_\lambda)^{-2d} \cong\C . \]
For a regular weight $\lambda$
\[ {\rm Ver}^\lambda \cong H^0(\widetilde{N}_\lambda, \wedge^0 T\widetilde{N}_\lambda)^{0} \cong\C . \]

Now consider the affine Springer fiber model. Each component 
$\ffl_{\bP_\lambda}^{t\gamma,\ell}$ contains a (partial) flag subvariety $X_\lambda = G/ P_\lambda$, where $P_\lambda$ is a parabolic subalgebra with the parabolic roots fixed by the stabilizer subgroup in $W$ of the weight $\lambda \in \Lambda$. Then the algebra $H^*(\ffl_{\bP_\lambda}^{t\gamma,\ell})^{\widetilde W \cdot}$ contains a well defined subalgebra  $H^* (X_\lambda)$ isomorphic to $C_W^{W_\lambda}$, which can be identified with $Z_{\rm HC}^\lambda$. The Higman ideal is then spanned by the socles of all blocks $C_W^{W_\lambda}$, and the Verlinde quotient is spanned by the heads of the regular blocks. 

Recall the isomorphism (\ref{block:DR}) 
of $W$-representations for the Springer action of $W$ on the left hand side  
\[ H^*(\ffl_{\bP_\lambda}^{t\gamma,\ell})^{\tW\cdot}\cong (\DR_W \otimes \sgn)\mathbf{e}_\lambda^- \;,  \]
that comes from Proposition \ref{prop:invariants}. Then $Z_{\rm HC}^\lambda \subset (\DR_W \otimes \sgn)\mathbf{e}_\lambda^- $  is a $W$-submodule isomorphic to $(C_W \otimes {\rm sgn})\mathbf{e}_\lambda^-$. Here $Z_{\rm Hig}$ is the subspace of the sign isotypical component in each block $(C_W \otimes {\rm sgn})\mathbf{e}_\lambda^- \subset (\DR_W \otimes \sgn)\mathbf{e}_\lambda^-$, and ${\rm Ver}$ is the trivial isotypical component of $(C_W \otimes {\rm sgn})\mathbf{e}_\lambda^-$ for regular $\lambda$. 

\[ Z_{\rm Hig} \cong\bigoplus_{[W_\lambda]} d_{\lambda, \ell} \; {\rm sgn}  \subset  Z_{\rm HC}  \cong \bigoplus_{[W_\lambda]} d_{\lambda, \ell} \;  (C_W \otimes \sgn)\mathbf{e}_\lambda^-  .\] 
Similarly we have 
\[ {\rm Ver} \cong\bigoplus_{[W_\lambda],\; \lambda \; {\rm regular}} d_{\lambda, \ell} \; {\rm triv}  \subset  Z_{\rm HC}  \cong \bigoplus_{[W_\lambda]} d_{\lambda, \ell} \;  (C_W \otimes \sgn)\mathbf{e}_\lambda^- .\] 

In particular, ${\rm dim} Z_{\rm Hig} = {\rm Cat}_W(\ell, h)$ is the number of orbits of the action of the extended affine Weyl group in the weight lattice, and  ${\rm dim} {\rm Ver} = {\rm Cat}_W(\ell-h, h)$ is the number of the regular orbits. Note that $Z_{\rm Hig}$ is exactly the isotypical component of the sign representation in $H^*(\ffl_{\bP_\lambda}^\gamma)^{\tW\cdot}$ with respect to the Springer action.

\section{Spectral curves and the parabolic Hitchin fibration}
\label{sec:hitchin}
As is now well-known, the cohomology of affine Springer fibers can be related to the cohomology of Hitchin fibers \cite{Ngo}. The cohomology of Hitchin fibers carries a natural perverse filtration, which can be used to obtain a bigrading on the space  $H^*(\Fl_\bI^\gamma)^{\tW\cdot} \cong H^*(\Fl_\bI^\gamma)^\Lambda$ for $G=SL_n$, which corresponds to the principal block of the center under Conjecture \ref{BBASV-conj}. In this section, we make this connection precise and conjecture that the obtained bigrading is that of the diagonal coinvariants up to a linear regrading. We note that by \cite[Conjecture 8.10]{KiTr}, Conjecture \ref{conj:perverseconjecture} is equivalent to Theorem \ref{thm:bgrd}, so that the bigradings obtained in the two ways should be equivalent.

There are several desirable features in obtaining the bigrading through the perverse filtration. For example, there is a natural Lefschetz element acting on the cohomology, coming from the relatively ample determinant bundle on the parabolic Hitchin fibration.  We conjecture that the $\fs\fl_2$-action on the center obtained this way coincides with the one given by the wedge product with the Poisson bivector field $\tau$ on the Springer resolution. 

\subsection{The parabolic Hitchin fibration}
First, we want to construct a particular compactification of the singular curve given by $x^n + y^n = 0 \subset \C^2$, inside a Hirzebruch surface. Most importantly, this compactification will be irreducible and have only an isolated singular point which is an ordinary $n$-uple point.

Let $\Sigma_r = \mathbb P(\mathcal O_{\mathbb P^1}(r) \oplus \mathcal O_{\mathbb P^1})$  be the $r$-th Hirzebruch surface. The Picard group of $\Sigma_r$ is generated by the zero section $E_r$ and the class of a fiber $F$, with intersection form determined by $F^2 = 0, E_r^2=-r$ and $E_rF=1$. 

Recall that there is a birational map from $\Sigma_r$ to $\Sigma_{r+1}$, called an ''elementary transform" (see \cite[Chapter 3]{Beauville}), constructed as follows. We choose some fiber $F$, 
and consider the surface $\Sigma_r’$, the blow-up of $\Sigma_r$ at $p := F \cap E_r$. Let $F’, E_r’$ be the 
strict transforms of $F, E_r$ and $\tilde E$ be the exceptional divisor of this blow-up. Then we have $$0 = F^2 = (F’+ \tilde E)^2 = (F’)^2 + 2 -1, $$ hence $F’$ is a $(-1)$-curve and can be contracted, the 
resulting surface being $\Sigma_{r+1}$. See Figure \ref{fig:1} for the toric picture, where the red line 
is the contracted curve.

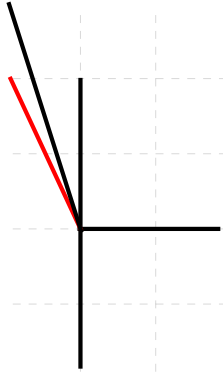
\begin{figure}[h]
\begin{center}
\begin{tikzpicture}
\draw[help lines, color=gray!30, dashed] (-0.9,-1.9) grid (1.9,2.9);
\draw[-,shorten <= -0.04cm,ultra thick] (0,0)--(2,0);
\draw[-,ultra thick] (0,-2)--(0,2.15);
\draw[-,shorten <= -0.03cm,ultra thick,color=red] (0,0)--(-1,2.15);
\draw[-,shorten <= -0.05cm,ultra thick] (0,0)--(-1,3.15);
\end{tikzpicture}
\caption{The toric blow-up and contraction giving a birational map $\Sigma_r\to \Sigma_{r+1}$.}
\end{center}
\label{fig:1}
\end{figure}
Now we can prove :

\begin{lemma}
\label{lemma:spectralcurve}
For all $n \geq 0$, there is an irreducible curve $C \subset \Sigma_2$ with a unique singularity analytically isomorphic to $x^n+y^n = 0$. \end{lemma}

\begin{proof}
Let $C_2 \subset \mathbb P^2$ be a smooth curve of degree $n$, and $\Sigma_1$ be the blow-up of
$\mathbb P^2$ at a point $a \notin C_2$. 
We denote by $C_1$ the preimage of $C_2$ in $\Sigma_1$. Consider a generic fiber $F_0$ and the corresponding elementary transform.

The strict transform $C’$  of $C_1$ inside $\Sigma_1’$ is isomorphic to $C_1$. Denote by $C$ the image of $C'$ under the contraction of $F'$. Since $F \cap C_1$ is given by $n$ points, we see that $C$ is analytically isomorphic to $C_2$ where $n$ points have been glued transversally together, resulting in an ordinary $n$-uple point $q$. It's clear that $C \backslash \{q\}$ is smooth. Since $C \backslash \{ q\}$ is connected, $C$ is irreducible. 
\end{proof}
\begin{remark}
Since $C_2$ is the normalization of $C$, the geometric genus of $C$ is $g_g=\binom{n-1}{2}$. Since the blowdown introduces $\binom{n}{2}$ nodes to $C'$, the arithmetic genus is 
$g_a=g_g+\binom{n}{2}=(n-1)^2$.
\end{remark}

\begin{definition}
\label{def:hitchinstack}
Let $X/\C$ be a smooth projective curve, $G$ a reductive group, and $\cL$ a line bundle on $X$ with $\deg \cL\geq g_X$.
The {\em Hitchin moduli stack} is the functor 
$$\cM: \text{Sch}_\C \to \text{Grpd}$$ sending $$S\mapsto \{(E,\varphi)| E \text{ is a } G\text{-torsor over } S\times X, \varphi\in H^0(\Ad(E)\otimes \cL)\}.$$
\end{definition}
\begin{definition}
\label{def:parabolichitchinstack}
Let $X,G,\cL$ be as above. The {\em parabolic Hitchin moduli stack}  is the functor 
$$\widetilde{\cM}:\text{Sch}_\C \to \text{Grpd}$$ sending 
$$S\mapsto \{(E,\varphi,x,E_x)|(E,\varphi)\in\cM, x\in X, E_x \text{ is a } B\text{-reduction along } \Gamma(x) \text{ of } E\}.$$
\end{definition}
Let $D$ be a divisor so that $\cO(D)=\cL$. The Hitchin moduli stack can be interpreted as classifying sections 
$$a: X\to \cO(D)\times^{\G_m} [\fg/G]$$
\begin{definition}
\label{def:hitchinfibration}
The morphism $$\cM\to \cA:=\bigoplus_{i=1}^n H^0(X, \cO(d_i D))$$
sending a section $a$ to its image in $\cO(D)\times^{\G_m} \Sym(\ft^*)^W$ is called the {\em Hitchin fibration}. The base $\cA$ is called the {\em Hitchin base}. The composition $$\widetilde{\cM}\to \cM\times X\to \cA\times X$$ is called the {\em parabolic Hitchin fibration}.
\end{definition}

Let now $G=SL_n$ and $\cL$ be a line bundle of degree $\geq 0$ on $\P^1$.
By the BNR correspondence \cite{BNR}, we may realize the curve $C$ from Lemma \ref{lemma:spectralcurve}, or rather its intersection with $\text{Tot}(\cO(2))$ as a spectral curve $\{\det(xI-\varphi)=0\}$ for the Hitchin fibration $$\cM\to \cA$$ associated to the data of $\P^1,G,\cL$. Let $a\in \cA$ be such that $C$ is the associated spectral curve. Note that we in fact have $a\in\cA^{ani}\subset \cA^{\heartsuit}$, the locus where the spectral curves are irreducible, resp. reduced (we will not need a more general definition of $\cA^{ani}$ or $\cA^\heartsuit$ here, for that see \cite[§ 6.1]{Ngo}).

The relationship to the affine Springer fibers considered in this paper is as follows. The curve $C$ may be chosen so that the unique singularity is over $0\in X$. Its local form corresponds to $\gamma=ts\in \fg(\cK)$ as before, for $s=\diag(1,\rho,\ldots,\rho^{n-1})$ where $\rho$ is a primitive $n$:th root of unity. Let
$(a,0)\in \cA^\heartsuit\times X$. Then \cite[Proposition 2.4.1]{Yun} says that 

\begin{equation}
\label{eq:localtoglobal}
\cP_a\times^{P^{red}_0(J_a)} \Fl^{\gamma}_\bI\to \widetilde{\cM}_a
\end{equation} is a homeomorphism of stacks.

Here $\cP_a$ is the generalized Picard stack, $P_0^{red}(J_a)$ the reduced quotient of the local Picard stack at $0$.
Modding out by $\cP_a$, the left-hand side of Eq. \eqref{eq:localtoglobal} simplifies to $\Fl^{\gamma}_\bI/P_0^{red}(J_a)$. By taking $\gamma=s t$ for $s\in \ft^{reg}$ as above, it is easy to compute by hand in this case that $P_0(J_a)=T(\C)\times \Lambda$ where $T$ is the diagonal torus in $\GL_n$ and $\Lambda=X^*(T)\cong \Z^n$ is the lattice part of the centralizer.

Modifying the proof of \cite[Proposition 4.13.1]{Ngo} slightly, we can write the following variant of Eq. \eqref{eq:localtoglobal}:
\begin{equation}
\label{eq:asfhitchin}
\widetilde{\cM}_a/\cP_a^\flat \cong \Fl^\gamma_\bI/\Lambda \;,
\end{equation} 
where $\cP_a^\flat$ is the Picard group of the normalization of $C$ as in \cite[4.7.3]{Ngo}. 

The upshot of this analysis is that we may define the {\em perverse filtration} on $H^*(\Fl^{\gamma}_\bI/\Lambda)$. Namely, 
if $\pi: \widetilde{\cM} \to \cM\times \{0\}\to \cA^{ani}$ denotes the restriction of the parabolic Hitchin fibration to the locus of irreducible spectral curves and with the parabolic reduction at $0\in X$, $\pi_*\C$ acquires a filtration from the $t$-structure on the base as $$P_{\leq i}:=\im({}^p\tau_{\leq i} \pi_*\C\to {}^p\tau_{\leq i+1} \pi_*\C).$$ Restricting to the stalk at $a$, we get a filtration $P_{\leq i}$ on $H^*(\widetilde{\cM}_a/\cP_a^\flat)\cong H^*(\Fl^{\gamma}_\bI/\Lambda)$. By results of Maulik--Yun \cite{MY} this filtration is independent of the choice of deformation of $C$ used here (we only require the total space to be smooth and a codimension estimate on the base, handled in this case by \cite{Ngo}). See also \cite[Section 3.1.3]{MY}.

Based on \cite[Conjecture 8.10]{KiTr} and results of the previous section, we make the following conjecture.
\begin{conjecture}
\label{conj:perverseconjecture}
For 
\begin{equation}
\label{eq:perverseconjecture}
\DR_n^{i,j}\cong \gr^P_{j+i} H^{2i}(\Fl^{\gamma}_\bI)^{\Lambda} \;.
\end{equation}
\end{conjecture}

\begin{proposition}
The conjecture \ref{conj:perverseconjecture} is true for $G = SL_2$.
\end{proposition}

\begin{proof}
In this case, the two vector spaces are equal to $\C^3$, hence we just need to check that the gradings agree. The affine Springer fiber $\Fl^{\gamma}_\bI$ can be identified with an infinite chains of $\mathbb P^1$, and the lattice action is obtained by translation by $2$ (\cite{Yun2}). Hence the quotient $X_0 = \Fl^{\gamma}_\bI/\Lambda$ is isomorphic to an elliptic curve with a singularity of type $I_2$ (i.e two $\mathbb P^1$ glued transversally twice). By the discussion before, this curve also appears as a spectral curve inside a cotangent bundle of $\mathbb P^1$, hence its compactified Jacobian is a Hitchin fiber inside the corresponding Hitchin fibration. Since $X_0$ has arithmetic genus $1$, it is isomorphic to its own compactified Jacobian. It follows by versality of the Hitchin map in this case that the restriction of this fibration to a generic line is simply a smoothing of $X_0$, say $f : X \to L = \C$. Let $L^* = L \backslash \{0\}$. By the decomposition theorem, we have $$f_* \underline{\C}_{X} = \underline{\C}_L \oplus  \underline{\C}_L[-2] \oplus  \underline{\C}_0[-2] \oplus \mathscr L[-1],$$ where $\mathscr L$ is the rank $2$ local system on $L^*$ given by the matrix $\begin{pmatrix} 1 & 2 \\ 0 & 1 \end{pmatrix}$. The pure part is given by  $\underline{\C}_L \oplus  \underline{\C}_L[-2] \oplus  \underline{\C}_0[-2]$. The perverse degrees are $-2,0,2$. Up to renormalisation, †his agrees with the diagonal coinvariants. 
\end{proof}

\subsection{The Lefschetz element}
Let $\cL_{det}$ be the determinant line bundle on $\cM$. The iterated cup product by $c_1(\cL_{det})$ induces a map 
$$\cup c_1(\cL_{det})^{g_a-i}: {}^pH^{\dim\cA+i}\pi_*\C\to{}^pH^{\dim\cA+2g_a-i}\pi_*\C \;,$$ and therefore maps 
$$\cup c_1(\cL_{det}): \gr^PH^*(\Fl^{\gamma}_\bI/\Lambda)\to \gr^PH^*(\Fl^{\gamma}_\bI/\Lambda)$$ of bidegrees $(a,b)=(2,2)$, where $a$ is the cohomological degree and $b$ is the perverse degree. 
  
We will now prove that $c_1(\cL_{det})$ coincides with a certain polynomial in the ring of diagonal coinvariants, under Conjecture \ref{conj:perverseconjecture}.
On the other hand, under the bigraded isomorphism of the principal block of the center $Z(\fu_\zeta^0)$ with sheaf cohomology groups of the Springer resolution, we can hope that  $c_1(\cL_{det})$ coincides with the Poisson bivector field on the Springer resolution as explained in more detail in Conjecture \ref{conj:springer}. 

\begin{theorem}
\label{thm:haimandet}
Under the identification Eq. \eqref{eq:perverseconjecture}, the element $c_1(\cL_{det})\in \gr^P H^*(\Fl^{\gamma}_\bI)^{\Lambda}$ corresponds up to a nonzero scalar to the ''Haiman determinant" 
$\Delta_{(n-1,1)}\in DR_n$ given by 
$$\Delta_{(n-1,1)}=\det(y_i^{p_j}x_i^{q_j})_{1\leq i,j\leq n} \;,$$
where $(p_1,q_1),\ldots, (p_n,q_n)$ is any ordering of $(0,0),(0,1),\ldots,(0,n-1),(1,0)\in\Z_{\geq 0}^2$.
\end{theorem}
\begin{proof}
Since the relevant subspace in grading $(n-1,n-1)$ is $1$-dimensional, contains $\Delta_{(n-1,1)}$ and $c_1(\cL_{det})$ is nonzero, we are done. 
\end{proof}

Finally, note that by the Jacobson--Morozov theorem, the nilpotent action of $e=\cup c_1(\cL_{det})$ extends to an $\fs\fl_2$-triple $(e,f,h)$ acting on $\gr^P H^*(\Fl^{\gamma}_\bI/\Lambda)$. By \cite[Conjecture 2.17]{MY} the Jacobson--Morozov filtration induced by $c_1(\cL_{det})$ on $H^*(\Fl^{\gamma}_\bI/\Lambda)$ is opposite to the perverse filtration. It is clear that the Jacobson--Morozov filtration induced by $\Delta_{(n-1,1)}$ on the diagonal coinvariants induces the filtration by antidiagonals.
Combining Theorem \ref{block:coherent} and Conjecture \ref{BBASV-conj}, we also have the following conjecture.
\begin{conjecture}
\label{conj:springer}
There is a bigraded algebra isomorphism
$$
\oplus_{i+j+k = 0} H^i(\widetilde{N}, \wedge^j T\widetilde{N})^k\cong H^*(\Fl^\gamma_\bI)^\Lambda \;,
$$
where the bigrading on the right is the one defined in in Theorem \ref{thm:bgrd}. Alternatively,
$$
\oplus_{i+j+k = 0} H^i(\widetilde{N}, \wedge^j T\widetilde{N})^k\cong \gr^P H^*(\Fl^\gamma_\bI)^\Lambda \;.
$$
Moreover, the element $\tau$ corresponding to the Poisson bivector field on the left should correspond up to a scalar to the polynomial $\Delta_{(n-1,1)}$ introduced in Theorem \ref{thm:haimandet}, or in the second version equivalently to $c_1(\cL_{det})$.
\end{conjecture}
In particular, combined with Theorem \ref{thm:bgrd} this Conjecture would imply \cite[Conjecture 4.9(3)]{LQ1}. 

\subsection{A degeneration of spectral curves}
In this section, we will geometrically construct a degeneration of spectral curves, including a map between the cohomologies of two affine Springer fibers in type $A$. The first of these cohomology groups is also conjecturally isomorphic to the double coinvariants by \cite{CO,KiTr}, and the other is the principal block of $H^*(\Gr^{\gamma,\zeta})^{\widetilde{W}\cdot}$, hence conjecturally equal to the $G^{\vee}$-invariant part of the center.  

We will study the elliptic homogeneous affine Springer fibers of slope $(n+1)/n$ associated to the elements 
$\gamma_{n+1/n}=u^{n+1}$, where 
$$u=\begin{pmatrix}
0  & \cdots & 0 & t\\
1  & \cdots & 0 & 0\\
0  & \ddots & 0 & 0\\
0 & 0 & 1 & 0
\end{pmatrix}.$$
In particular, we study these affine Springer fibers in relation to $\Fl^{\gamma}_\bI$, where $\gamma=ts$ as in the previous section. 
In particular, we will construct a a family of irreducible spectral curves $C_t \subset \mathrm{Tot}(\mathcal O_{\mathbb P^1}(2))$, such that the associated family of parabolic Hitchin fibers models the degeneration of affine Springer fibers of slope $n/n$ to the one of slope $(n+1)/n$. One can then ask whether the specialization map from the cohomology of the total family (which is just that of the central fiber) gives an injection to the cohomology of the special fiber, respecting the perverse filtration.

\begin{theorem}
\label{thm:spectral}
There exists a family of irreducible curves $\cC\to \A^1$, arising as a restriction of the Hitchin system to a line on the Hitchin base, so that the spectral curve $C_t, t\in \A^1$ will have two singular points: one singular point with equation $y^n=z^{n-1}$ independently of $t$, and another singular point of the form $y^n = tx^n + x^{n+1}$. 
\end{theorem}
\begin{proof}
We construct the family of spectral curves realizing  this degeneration as follows: let $E \subset \Sigma_1$ be the exceptional section inside the first Hirzebruch surface, and $F \subset \Sigma_1$ some fiber, which we will call "the fiber at infinity". Let $U = \Sigma_1 \backslash ( E \cup F)$. Take coordinates $x,y$ on $U$ such that the straight lines $x = \text{constant}$ are the fibers of the projection $U \subset \Sigma_1 \to \mathbb P^1$. Let us consider the curve $\hat{C}_t \subset U$ given by the equation $y^n = t + x$. The effect of a positive elementary transform $\varphi : \Sigma_r \dashrightarrow \Sigma_{r+1}$ is given by the change of variables $u = y/x, v = x$. 

Hence the strict transforms of $\hat{C}_t$ (inside $\varphi(U)$) have local equation given by $u^n = tv^n + v^{n+1}$, giving the desired degeneration. Now let us describe the singular point at infinity (i.e compute the closure of these curves inside $\Sigma_2$), and prove that $C_t$ is irreducible for all $t \in \mathbb A^1$. 

First, we claim that the closure of $\hat{C}_t$ doesn’t intersect $E$. Indeed, recall that $\Sigma_1$ is the blow-up of $\mathbb P^2$ at a point. Hence, it’s enough to take the closure of the preimage of $\hat{C}_t$ inside $\mathbb P^2$ (call this curve $\tilde C_t$) and check that $\tilde C_t$ doesn’t intersect the center of the blow-up. On $U$, we have coordinates $x,y$, that form a dense open of $\mathbb P^2$ (recall that $U$ and $E$ are disjoint by definition). Because $U \cong \mathbb A^2$, we can take homogeneous coordinates $[x:y:z]$ on $\mathbb P^2$. The fiber at infinity is given by $z=0$ and $U$ is given by $z=1$. The fiber $x=0$ and $z=0$ both contains the center of the blow-up which is therefore $[0:1:0]$. 

The closure of $\tilde C_t$ has equation $y^n = tz^n + xz^{n-1}$, which clearly doesn’t contain $[0:1:0]$. Since the elementary transforms are isomorphism outside the exceptional locus, it follows that the closure of $C_t$ coincide with the closure of $\tilde C_t$ inside $\mathbb P^2$, i.e the curve with equation $y^n = tz^n + xz^{n-1}$. The only point at infinity is $[1:0:0]$, and has local equation $y^n = tz^n + z^{n-1}$ as claimed. To check that $C_t$ is irreducible, it's enough to check that $C_t$ is irreducible on the chart $x \neq 0$. On this chart, $C_t$ is isomorphic to the curve given by $y^n = z^{n-1}$, which is irreducible. 
\end{proof}
Consider the associated family of parabolic Hitchin fibers, which is a restriction of the family in Definition \ref{def:hitchinfibration} to a line. Using Eq. \eqref{eq:localtoglobal}, we note that the only affine Springer fibers contributing to the cohomology are the ones coming from the singularities described above. We will ignore the one which is constant, for there is an injective map in cohomology sending the cohomology classes $\alpha\in H^*(\Fl_\bI^\eta/\Lambda_{\eta})$ of interest to $$ \alpha\otimes 1 \in H^*(\Fl_\bI^\eta/\Lambda_\eta)\otimes H^*(\Fl_\bI^{\gamma_{n-1/n}})\cong H^*(\widetilde{\cM}_a) \;,$$ where $\eta$ is either $\gamma$ or $\gamma_{n+1/n}$ depending on $t$.
\begin{remark}
If $t \neq 0$, note that $y^n=tx^n+x^{n+1}$ is locally isomorphic to the singularity $y^n + x^n = 0$. 
\end{remark}
In particular, by Theorem \ref{thm:spectral}, we get a pullback map $i^* : H^*(\Fl_\bI^{\gamma_{n+1/n}})\to H^*(\Fl^\gamma_\bI/\Lambda)$, coming from the inclusion $i : C_t \to \mathcal C$ of a fiber with $t\neq 0$ into the family $\mathcal C$. \begin{remark}
The perverse filtration on both the source and the target of $i^*$ is defined using the perverse filtration on the full Hitchin fibration. However, it is unclear how the perverse filtration defined in the full family compares with that induced by the $t$-structure on $\A^1$, as the pullback along the inclusion to the base is, in general, only right $t$-exact.
\end{remark}
We end the paper with the following conjecture.
\begin{conjecture}
The map $i^*$ is injective and its image is exactly $H^*(\Fl^\gamma_\bI)^\Lambda$. Moreover, since the map respects the perverse filtrations, we have an isomorphism of bigraded vector spaces 
$$\gr^PH^*(\Fl^\gamma_\bI)^\Lambda\cong \gr^PH^*(\Fl^{\gamma_{n+1/n}}_\bI).$$ Note that together with Conjecture \ref{conj:perverseconjecture} and the results of \cite{CO}, this is compatible with Theorem \ref{thm:bgrd}.
\end{conjecture}

\end{document}